\documentclass[12pt]{article}
\usepackage{graphicx}

\usepackage{amsmath, amssymb, amsthm}

\newtheorem{thm}{Theorem}[section]

\newtheorem{lem}[thm]{Lemma}
\theoremstyle{definition}
\newtheorem{df}[thm]{Definition}

\newtheorem*{rem*}{Remark}
\newtheorem{prop}[thm]{Proposition}

\numberwithin{equation}{section}

\def\R{\mathbb{R}}
\def\Ri{\mathbb{R}\cup \{+\infty\}}

\def\pa{\partial }
\def\dom{\mathrm{dom}\,}

\def\dist{\mathrm{dist}\,}
\def\e{\varepsilon}
\def\eps{\varepsilon}

\def\ol{\overline}

\def\ba{\begin{array}}
\def\ea{\end{array}}
\def\be{\begin{equation}}
\def\ee{\end{equation}}

\title{A primal approach to the Clarke-Ledyaev inequality\footnote{The research of M. Hamamdjiev and M. Ivanov  is supported by the European Union-NextGenerationEU, through the National Recovery and Resilience Plan of the Republic of Bulgaria,  project  SUMMIT BG-RRP-2.004-0008-C01. The research of N. Zlateva is supported by the Scientific Fund of Sofia University under grant  80-10-138/17.04.2024.}}
\author{Mihail Hamamdjiev\thanks{Sofia University "St. Kliment Ohridski", Faculty of Mathematics and Informatics, 5 James Bourchier Blvd., 1164 Sofia, Bulgaria, e-mail:mihailh@fmi.uni-sofia.bg}, Milen Ivanov\thanks{Radiant Life Technologies Ltd., Nicosia, Cyprus, e-mail:milen@radiant-life-technologies.com}, Nadia Zlateva\thanks{Sofia University "St. Kliment Ohridski", Faculty of Mathematics and Informatics, 5 James Bourchier Blvd., 1164 Sofia, Bulgaria, e-mail:zlateva@fmi.uni-sofia.bg}}
\date{December, 2024}

\begin{document}

\maketitle

\begin{abstract}
  We present a version of the Clarke-Ledyaev inequality that does not involve elements of the dual space. The proof  relies mainly on geometry and on the classical lemma of Bishop and Phelps. In addition, this approach allows us to provide a simplified proof of the Clakre-Ledyaev inequality. The  approach  is primal in the sense that    no dual arguments are used. 
\end{abstract}

\section{Introduction}
    \label{sec:intro}
    The purpose of this work is to present a streamlined and more geometrical proof of Clarke-Ledyaev multidirectional mean value inequality, relating the latter to the classical Bishop-Phelps Lemma.

    In their original work \cite{CL} F. Clarke and Yu. Ledyaev, after having proved their famous inequality for the Clarke subdifferential, used it to establish estimates for the multidirectional Subbotin derivative. As the elements of a subdifferential belong to the dual space,   the Clarke-Ledyaev inequality intrinsically involves dual notions. On the other hand,  Subbotin derivative,
    see \cite[p. 313]{CL}, is a purely primal object, so one can anticipate that its properties  should be derived  via primal concepts.

    Here we reverse the approach used above: first we prove theorems of Rolle's and Lagrange's type  for a multidirectional derivative, see \eqref{eq:d-def} for its definition. It is similar to the Subbotin derivative, but it has a slightly simpler form. From there  we derive the Clarke-Ledyaev inequality. Note that the Rolle's type multidirectional inequality we establish is very precise and it is equivalent to the classical Bishop-Phelps Lemma, see e.g. \cite{bp}. This gives a distinct geometrical flavour to our proof.

    We work in a Banach space $(X,\|\cdot\|)$ with a closed unit ball $B_X$. We denote by $B(x;\varepsilon)$ the closed ball of radius $\varepsilon$ around $x$, that is, $B(x;\varepsilon)=\{y\in X:\ \|y-x\|\le\varepsilon\}=x+\varepsilon B_X$. The open ball of  radius $\eps $ around $x$ is denoted by $B^\circ(x;\varepsilon)$.
    For $a\in X$ and a nonempty  set $A\subset X$,
    $$
        [a,A] := \{a+t(x-a):\ t\in[0,1],\ x\in A\}
    $$
    denotes the multidirectional interval between $a$ and $A$;
    $$
    A\pm a:=\{ x\pm a\colon x\in A\}
    $$
denotes the sum (resp. the difference) of the sets $A$ and $\{a\}$, and
    $$
        A_\varepsilon := A + \varepsilon B_X
    $$
denotes the $\varepsilon$-enlargement of $A$.

For an extended real-valued function $f:X\to \Ri$ and a non-empty set $A\subset X$ we denote by $\inf f(A):=\inf_{a\in A} f(a)$.

For a function $f:X\to \Ri$  and a nonempty set $ A\subset X$ we define the  $A$-\emph{directional derivative of} $f$ \emph{at} $x\in \dom f$ as
    \begin{equation}
        \label{eq:d-def}
        f^-(x;A) := \liminf_{t\downarrow 0} \frac{ \inf f(x+tA)-f(x)}{t}.
    \end{equation}

Using the classical Bishop-Phelps Lemma, eg. cf. \cite[Lemma 1.2]{bp}, we can give a simple proof of the following primal Clarke-Ledyaev inequality. As it is seen from the proof, we can actually obtain a stronger result, but we prefer to keep the differential form of the inequality for the sake of consistency.
\begin{thm}
    \label{thm:primal-cled}
    Let $A\subset X$ be a non-empty, closed, convex and bounded set. Let $f:X\to\Ri$ be a lower semicontinuous function and let  $a\in\dom f$. If for some $r\in\R$
    \begin{equation}
        \label{eq:cled-coincides}
        r \le  \inf  f(A),
    \end{equation}
    then there exists $\bar x\in[a,A]$ such that $f(\bar x) \le \max\{f(a),r\}$, and
    \begin{equation}
        \label{eq:prim-Lagr}
        f^-(\bar x;A-a) \ge r - f(a).
    \end{equation}
\end{thm}
Note that if $\dom f \cap A \neq\varnothing $, then \eqref{eq:prim-Lagr} can be rewritten in a nicer form:
$$
    f^-(\bar x;A-a) \ge  \inf f(A)-f(a).
$$

Via a duality argument we can establish the celebrated Clarke-Ledyaev inequality as a corollary to Theorem~\ref{thm:primal-cled}. Before   formulate the Clarke-Ledyaev inequality, we first have to give the notion of a \textit{subdifferential} of a function defined on $X$. By the efforts of several researchers during the years, notably A. Ioffe (see \cite{ioffe-1} for a summary of the state-of-the-art  at that time) and L. Thibault~\cite{Tbook}, it is now apparent that the most efficient way to deal with subdifferentials is to work in an abstract framework in which the subdifferentials are defined as  multivalued functions satisfying a list of basic properties described by a set of axioms, rather than to list all the particular  subdifferentials for which a given statement holds.  According to this, we give the notion of a \textit{feasible subdifferential} of a function defined on $X$ before   formulate the Clarke-Ledyaev inequality.

\begin{df}
    \label{def:subdef}
    Let $(X,\|\cdot\|)$ be a Banach space and let $f:X\to\Ri$ be a lower semicontinuous function. A subdifferential operator $\pa$ applied to $f$ at a point $x$ produces  a multivalued map $\pa f:X\to2^{X^*}$. The subdifferentail $\pa$ is said to be \emph{feasible} on $X$, if it satisfies:

    (P1) $\pa f(x)=\varnothing$ for $x\notin\dom f$;

    (P2) $\pa f(x)=\pa g(x)$ whenever $f\equiv g$ on a neighbourhood of $x$;

    (P3) If $f$ is convex and continuous on a neighbourhood of $x$, then $\pa f(x)$ coincides with the canonical subdifferential in the sense of Convex analysis;

    (P4) If $f$ attains a local minimum at $x\in\dom f$, then $0\in\partial f(x)$;

    (P5) If $g$ is convex and continuous and $f+g$ attains its minimum at $\bar x \in X$, then for any $\varepsilon >0$ there exist $p\in\pa f(x)$ and $q\in\pa g(y)$ such that
    \[
    x,y\in B^\circ(\bar x;\varepsilon), \ |f(x)-f(\bar x)| < \varepsilon, \text{ and } \|p+q\| <\varepsilon;
    \]

(P6) There is a   subdifferentail $\tilde\partial$ on $X\times \R$ satisfying (P1)--(P5) such that if $g:X\times\R\to\Ri$ is such that $g(x,t)=f(x)+\kappa t$ where $\kappa \in\R$, then
    \[
        \tilde\pa g(x,t)\subseteq\pa f(x)\times\{\kappa\},\quad\forall x\in X,\ \forall t\in\R.
    \]
\end{df}
 The  definition above which includes only properties (P1)--(P5)  is essentially that found at \cite[p.~679]{Tbook}. On the same page the author shows how all known subdifferentials satisfy this set of axioms. In this sense, the given definition of a feasible subdifferential is a very general one. In  \cite[Definition~2.1]{ioffe-2}, see also \cite[p.~152]{ioffe-book}, the property (P1) is called \textit{substantiality}, (P2) is \textit{localisability}, (P3) is \textit{contiguity}, and (P4) is \textit{optimality}. The \textit{fuzzy sum rule} (P5) is not considered in \cite{ioffe-2} as a basic property, which is natural, because for many subdifferentials it is tough to prove it, and a special definition \cite[Definition~2.12]{ioffe-2} is devoted to it. Note that the sum rule in \cite[Definition~2.12]{ioffe-2} is stronger than (P5).

It is shown in \cite{HI} that any   subdifferential satisfying (P1)--(P5) of Definition~\ref{def:subdef} satisfies the Clarke-Ledyaev inequality and the interested reader can use the auxiliary function $\varphi_K$ from \cite{HI} (instead of the distance function that appears in the proof of Proposition~\ref{prop:bridge}), to derive the Clarke-Ledyaev inequality from Theorem~\ref{thm:primal-cled}. We, however, prefer to follow the much shorter route of \cite{ACL} and to this end we need  the additional assumption (P6). Nevertheless, it is a very basic property --  namely (S5b) of \cite[Definition~2.1]{ioffe-2} --  a part of \textit{calculability}. It is indeed immediately checked for all subdifferentials that we know of, but it is a bit harder to formalize, because it "goes" outside of the underlying space. Here it is reasonable to heed \cite[Remark 2.3]{ioffe-2}: \textit{"Strictly speaking, in a definition of a subdifferential we have to mention the class of functions and spaces on which it is defined or considered. In many cases however definitions make sense for all functions on all Banach spaces".}

The variant of Clarke-Ledyaev inequality we derive in this work is as follows.
\begin{thm}
    \label{thm:DualCLI}
    Let $(X,\|\cdot\|)$ be a Banach space. Let $\pa$ be a feasible subdifferential on $X$.

    Consider a non-empty, convex, bounded and closed set $A\subset X$ and a point $a\in X$. Let $f:X\to\Ri$ be a lower semicontinuous function such that $a\in\dom f$, and let $r\in\R$ be such that
    \begin{equation}
        \label{eq:cled-cond}
        r < \sup_{\delta>0} \inf f(A_\delta).
    \end{equation}

    Let $\e>0$ be arbitrary. Then there exist $\xi\in[a,A]_{\e}$ such that
    \begin{equation}
        \label{eq:f-xi-local}
        f(\xi) < \max\{f(a),r\} + \varepsilon,
    \end{equation}
    and $p\in\pa f(\xi)$ such that
    \begin{equation}
        \label{eq:cl-p-A}
        \inf p(A)-p(a) > r - f(a).
    \end{equation}
\end{thm}

Further on, the work is organized as follows. In the next Section~\ref{sec:primal} we prove Theorem~\ref{thm:primal-cled}. Then in Section~\ref{sec:dual} we establish the needed connection  between the multidirectional derivative and the feasible subdifferential. In Section~\ref{sec:proof} we derive the Clarke-Ledyaev inequality. Finally,  is an Appendix we provide a proof of the Bishop-Phelps Lemma using  primal arguments.

\section{Primal result}
\label{sec:primal}

Note that if $\varnothing \neq A\subset X$ is closed and bounded set, then the interval $[a,A]$ is also closed and bounded.

    Indeed, we can assume without loss of generality that $a=0$ and let the sequence $(t_nx_n)_{n=1}^\infty$, where $t_n\in[0,1]$ and $x_n\in A$, be convergent to $\alpha$. Let $t_{n_k}\to t$ as $k\to\infty$. If $t=0$, then as the sequence $(x_n)_{n=1}^\infty$ is bounded, we have that $\alpha =0$, hence $\alpha \in[0,A]$. If, on the other hand, $t\in (0,1]$, then $x_{n_k} = (t_{n_k}x_{n_k})/t_{n_k} \to \alpha/t$, as $k\to\infty$ and $\alpha/t\in A$, because the latter is closed.

    Another geometrical notion, related to the interval $[a,A]$, is the cone generated by $A$ with apex $a$:
    $$
        C(a;A) := \{a+t(x-a):\ t\ge 0,\ x\in A\}.
    $$
    Obviously, $C(a;A) = \{a\} + C(0;A-a)$.

  Let us recall the classical Bishop-Phelps Lemma, see eg.  \cite[Lemma 1.2]{bp}. By drawing a picture the reader will imediately recognise the geometry within the standard proof of Ekeland Variational Principle, eg. cf. \cite[p.45]{phelps}.
    \begin{lem}[Bishop-Phelps]
        \label{lem:bp}
        Let $X$ be a Banach space and let $A$ be a non-empty, closed, convex and bounded subset of $X$. Let $a\in X\setminus A$.

        Let $M\subset X$ be a closed set such that $a\in M$, and $C(a;A)\cap M$ is bounded.

        Then there exists $\bar x \in M \cap C(a;A)$ such that
        \begin{equation}
            \label{eq:bp-concl}
            M\cap \left(\bar x + C(0;A-a)\right) = \{\bar x\}.
        \end{equation}
    \end{lem}

Lemma \ref{lem:bp} is purely primal in its statement as it involves only notions from the underlying space. However, in the proof provided in \cite{bp} is used  a separation theorem, which of course involves elements from the dual space. To use only primal arguments in our proof, we give an alternative proof of this result in the Appendix.	

    By using the Bishop-Phelps Lemma, we readily get the "Rolle" part of the primal Clarke-Ledyaev inequality.
    \begin{prop}
        \label{prop:rol-prim}
        Let $A\subset X$ be a non-empty, closed, convex and bounded set. Let $f:X\to\Ri$ be a proper lower semicontinuos function.
        If
        \begin{equation}
            \label{eq:rol-prim}
            f(a) \le \inf f(A),
        \end{equation}
        then there exists $\bar x\in[a,A]$ such that $f(\bar x) \le f(a)$, and
        $$
            f^-(\bar x;A-a) \ge 0.
        $$
    \end{prop}
    \begin{proof}
        Assume first that
        \begin{equation}
            \label{eq:f-a-less}
            f(a) < \inf f(A),
        \end{equation}
        which implies that $a\in \dom f$ and $a\not \in A$.
        Consider the set
        $$
            M := \{x\in [a,A]:\ f(x) \le f(a)\}.
        $$
        The set $M$  is closed, as the function $f$ is lower semicontinuos and the set  $[a,A]$ is closed. $M$ is bounded as $A$ is bounded, and $M$ is non-empty, as $a\in M$ by definition. By Bishop-Phelps Lemma~\ref{lem:bp} there is $\bar x \in [a,A]$ such that \eqref{eq:bp-concl} holds.
        The latter means that $f(\bar x) \le f(a)$ and thus $f(x) > f(a)$ for all $x\in [a,A]\setminus\{\bar x\}$ such that $x\in \bar x + C(0;A-a)$.

        But \eqref{eq:f-a-less} and $f(\bar x) \le f(a)$ imply that $\bar x \not\in A$ and then it is easy to see that there is $\delta > 0$ such that $B(\bar x;\delta) \cap \left(\bar x + C(0;A-a)\right) \subset [a,A]$. So, $f(\bar x) \le f(x)$ for all $x$ in $B(\bar x;\delta) \cap \left(\bar x + C(0;A-a)\right)$. Directly from the definition of the multidirectional derivative, see \eqref{eq:d-def}, it follows that $f^-(\bar x; A-a) \ge 0$.

        Turning now to the general case, if $f(a)\le \inf f(A)$, then again directly from the definition \eqref{eq:d-def} it follows that $f^-(\bar x; A-a) \ge 0$. Otherwise, there will be some $a'\in[a,A]$ such that $f(a')<\inf f(A)$ and from the first part of the proof there will be  $\bar x \in [a',A]\subset[a,A]$ such that $f(\bar x) \le f(a') < f(a)$, and $f^-(\bar x; A-a') \ge 0$. But by convexity of $A$ it is easy to check that $f^-(\bar x; A-a) \ge f^-(\bar x; A-a')$.
    \end{proof}
    We will use the ingenious construction from \cite{ACL} to move from Rolle type of result to Lagrange type of result  and to prove the primal Clarke-Ledyaev inequality.

    \begin{proof}[\textbf{\emph{Proof of Theorem~\ref{thm:primal-cled}}}] Considering instead of $f$ the function $x\to f(x) - f(a)$, we may and do assume for convenience that $f(a)=0$.

        Consider the Banach space $\widetilde{X}:=X\times\R$, the
        point $\widetilde{a}:= (a,0)$, the
        set $\widetilde{A}:=A\times\{1\}=(A,1)\subset\widetilde{X}$ and  the function $\widetilde{f}:\widetilde{X}\to\Ri$ defined by
        $$
            \widetilde{f}(x,t) :=f(x)-rt,\quad\forall x\in X,\ \forall t\in \R.
        $$
        It is clear that $\widetilde{f}$ is lower semicontinuous. Furthermore,
        $$
         \widetilde{f}(\widetilde a) =   \widetilde{f}(a,0)=f(a) = 0,
        $$
        and
        $$
         \inf   \widetilde{f}(\widetilde{A}) \ge \inf f(A) - r  \ge 0 = \widetilde f(\widetilde{a}).
        $$
        Applying  Proposition~\ref{prop:rol-prim} to $\widetilde f$, $\widetilde A$, and $\widetilde a$ one gets  $(\bar x,\bar t)\in [(a,0),(A,1)]$ such that  $\tilde f (\bar x , \bar t) \le 0$, i.e. $f(\bar x) \le r\bar t \le \max\{0,r\}$, and
        $$
             \widetilde{f}^-((\bar x,\bar t),\widetilde{A}-\widetilde{a})\ge 0.
        $$
        But $\widetilde{f}^-((\bar x,\bar t),\widetilde{A}-\widetilde{a})=\widetilde{f}^-((\bar x,\bar t);(A,1)-(a,0)) =  f^-(\bar x;A-a) - r$.
    \end{proof}

    \section{Duality}
    \label{sec:dual}
    Establishing relations between primal and dual differential notions is a central thread in Analysis and in Variational Analysis. In this section we develop the tools we need to "bridge" the multidirectional derivative and the feasible subdifferential.

   First we recall an equivalent form of \cite[Proposition 2.1]{HI}.
    \begin{lem}\emph{\cite[Proposition 2.1]{HI}}
        \label{lem:p4}
        Let $X$ be a Banach space. Let $f:X\to\Ri$ be a proper and lower semicontinuous function. Let $\pa$ satisfies properties (P1)--(P5) of the definition of a feasible subdifferential on $X$. Let $g:X\to\R$ be a convex continuous function such that $f+g$ be bounded below. Then for each $\varepsilon>0$ there are $x,y\in X$, $p\in\pa f(x)$ and $q\in\pa g(y)$ such that
        $$
            \|x-y\|<\varepsilon, \ f(x)+g(y) < \inf (f+g)(X) + \varepsilon, \text{ and } \|p +q\| < \varepsilon.
        $$
    \end{lem}
   Further, we continue with the following simple
    \begin{lem}
        \label{lem:dir-der-ineq}
        Let $A\subset X$ be a non-empty, closed, convex and bounded set. If for some $x\in X$
        $$
            f^-(x;A) > 0,
        $$
        then there exist $k,\overline \varepsilon > 0$ such that
        \begin{equation}
            \label{eq:dir-der-ineq}
            f(y) - f(x) \ge k\|y-x\|,\quad\forall y\in (\{x\}+C(0;A)) \cap B(x;\overline\varepsilon).
        \end{equation}
    \end{lem}
    \begin{proof}
        Let $\lambda > 0$ be such that $\lambda < f^-(x;A)$. By the definition of the multidirectional derivative of $f$ at $x$ there is $\delta>0$ such that
        \begin{equation}
            \label{eq:t-lambda}
          \inf  f(x+tA) > f(x) + t\lambda,\quad\forall t\in(0,\delta).
        \end{equation}
        Obviously, this implies that $0\not\in A$. Since $A$ is closed, there is $\mu$ such that
        $$
           0< \mu < \inf_{a\in A} \|a\|.
        $$
        Set $s:=\sup_{a\in A} \| a\|$ and observe that $s$  is finite as $A$ is bounded. Set
        $$
         \overline\varepsilon := \mu\delta,\text { and }k := \lambda/s,
        $$
        and let $y\in (\{x\}+C(0;A)) \cap B(x;\bar{\e})$ be arbitrary. Since $y-x\in C(0;A)$, there are $a\in A$ and $t\ge0$ such that $y-x = ta$. But $\bar{\e} \ge \|y-x\| = t\|a\| > t\mu$. Then $t < \bar\e/\mu = \delta$.

        If $t=0$, that is $y=x$, then \eqref{eq:dir-der-ineq} is trivially satisfied, and if $t>0$ then \eqref{eq:t-lambda} implies $f(y) - f(x) > \lambda t$. But $t = \|y-x\|/\|a\| \ge \|y-x\|/s $ and  \eqref{eq:dir-der-ineq} again holds.
    \end{proof}

    The following key proposition essentially bridges the multidirectional  derivative defined via  \eqref{eq:d-def}, and the feasible subdifferential. Even a stronger inequality  can be derived, but we wish to keep the technicalities down to the bare minimum and thus we prove the following version, which is enough for our purposes.

    \begin{prop}
        \label{prop:bridge}
        Let $X$ be a Banach space. Let $f:X\to\Ri$ be a proper and lower semicontinuous function. Let $\pa$ be a feasible subdifferential on $X$.
        Let $A\subset X$ be a non-empty, convex, bounded and closed set. Assume that for some $\bar x\in\dom f$ and some $\delta>0$
        \begin{equation}
            \label{eq:bridge-cond}
            f^-(\bar x;A_\delta ) > 0.
        \end{equation}
        Then for each $\varepsilon > 0$ there are $x\in B^\circ(\bar x;\varepsilon)$ such that $|f(x) - f(\bar x)| < \varepsilon$, and $p\in\partial f(x)$ with
        \begin{equation}
            \label{eq:bridge-concl}
           \inf p(A ) > - \varepsilon.
        \end{equation}
    \end{prop}

    \begin{proof}
        Let $\varepsilon \in (0,1)$ be fixed. Since $f$ is a lower semicontinuous function, there is $\gamma \in  (0,\eps)$ such that
        \begin{equation}
            \label{eq:f-eps-bound}
            f(B(\bar x;\gamma))> f(\bar x) - \varepsilon.
        \end{equation}
        Next, by \eqref{eq:bridge-cond} and Lemma~\ref{lem:dir-der-ineq} it follows that there is some $k>0$ and $\overline \eps$ such that
        \begin{equation*}
                     f(y) - f(\bar x) \ge k\|y- \bar x\|,\quad\forall y\in (\{ \bar x\}+C(0;A_\delta) )\cap B(\bar x;\overline \eps).
        \end{equation*}
        By taking $\gamma$  so small that $\gamma <\overline \eps$, we have that
         \begin{equation}
            \label{eq:f-bound-bound}
            f(y) - f(\bar x) \ge k\|y- \bar x\|,\quad\forall y\in (\{ \bar x\}+C(0;A_\delta) )\cap B(\bar x;\gamma).
        \end{equation}

        Consider the convex  continuous function
        $$
            \psi_n(x) := n.\dist(x, \{ \bar x\}+C(0;A )),
        $$
        where $n\in\mathbb{N}$. We claim that
        \begin{equation}
            \label{eq:q-le-0}
         \inf   q( -A) \ge 0,\quad\forall x\in X,\ \forall q\in \pa \psi_n(x).
        \end{equation}
        Indeed, fix arbitrary   $q\in \pa \psi_n(x)$ and $a\in A$. Due to (P3) we have $q(a ) \le \psi_n(x + a) - \psi_n(x)$. But if $y\in \{ \bar x\}+C(0;A )$, say $y=\bar x + tb$ for some $b\in A$ and $t\ge 0$, is such that $\|x-y\| < \dist(x, \{ \bar x\}+C(0;A )) + \delta$ for some $\delta >0$, then
        $$
            y+a = \bar x + tb + a = \bar x + (1+t)\left[\left(\frac{1}{1+t}a+\frac{t}{1+t}b\right)\right]\in  \{ \bar x\}+C(0;A ),
        $$
        by convexity of $A$. Therefore, $\dist(x+a , \{ \bar x\}+C(0;A )) \le \| x+a-(y+a)\|=\|x-y\| < \dist(x, \{ \bar x\}+C(0;A )) + \delta$, and since $\delta>0$ is arbitrary, we have $\psi_n(x + a) - \psi_n(x) \le 0$. Hence, $q(a)\le 0$ and by the choice of $a$ it follows $\inf q( -A)\ge 0$, thus \eqref{eq:q-le-0} is verified.

        Next, let
        $$
            \widehat f(x) := \begin{cases}
                f(x),\quad x\in B(\bar x;\gamma),\\
                +\infty,\quad x\not\in B(\bar x;\gamma).
            \end{cases}
        $$

        It is clear that $\widehat f$ is proper, lower semicontinuous and bounded below by $f(\bar x) - \varepsilon$, see \eqref{eq:f-eps-bound}.

        Consider for $n\in\mathbb{N}$ the functions
        $$
            \widehat f + \psi_n.
        $$
        From Lemma~\ref{lem:p4} (applied to  $\widehat f + \psi_n$ and $\varepsilon=1/n$),  we get $x_n,\ y_n\in X$,  $p_n\in\partial \widehat f(x_n)$ and $q_n\in \partial \psi_n(y_n)$ such that
        \begin{equation}
            \label{eq:x-y-close}
            \|x_n-y_n\|<1/n,
        \end{equation}
        \begin{equation}
            \label{eq:f-psi-to-inf}
            \widehat f(x_n) + \psi_n(y_n) < \inf (\widehat f + \psi_n )(X)+ 1/n \le f(\bar x) + 1/n,
        \end{equation}
        and $\|p_n + q_n\|< 1/n$. Since $A$ is bounded, $  s:=\sup_{a\in A}\| a \|$ is finite. So, for any $a\in A$,
        \[
        p_n(a )>-q_n(a ) -\frac{1}{n}\|a \| > q_n( -a) - \frac{1}{n}  s,
        \]
        and then
        \[
       \inf p_n(A )\ge \inf q_n(   -A)-\frac{1}{n}  s \ge -\frac{1}{n} s.
        \]
 Hence, there is an $N\in\mathbb{N}$ such that
        $$
         \inf   p_n(A ) > -\varepsilon,\quad\forall n > N.
        $$

To conclude, we need to prove the following

        \smallskip
        \noindent
        \textsc{Claim.} There are arbitrarily large $n$'s such that $x_n\in B^\circ(\bar x;\gamma)$.

        \smallskip
        \noindent
        Assume the contrary: there is $N_1\in\mathbb{N}$ such that
       \begin{equation}
            \label{eq:y-to-boundary}
            \|x_n - \bar x\| = \gamma,\quad\forall n > N_1.
      \end{equation}
        From \eqref{eq:x-y-close} it immediately follows that
      $$
            \lim_{n\to\infty}\|y_n-\bar x\| = \gamma.
       $$
        From \eqref{eq:f-psi-to-inf} and \eqref{eq:f-eps-bound} we get
        $$
            \psi_n(y_n) \le f(\bar x) - f(x_n) +1/n \le \eps +1/n\le 2, \quad\forall n\in\mathbb{N}.
        $$
        From the definition of $\psi_n$ we then have that
        $$
            \lim_{n\to\infty} \dist(y_n,\{ \bar x\}+C(0;A ) ) = 0,
        $$
        and by \eqref{eq:x-y-close},
    $$
            \lim_{n\to\infty} \dist(x_n, \{ \bar x\}+C(0;A )) = 0.
        $$

        So, there are $a_n\in A$ and $t_n\ge 0$ such that
        \begin{equation}
            \label{eq:tn-an}
            \lim_{n\to\infty} \|x_n - \bar x - t_na_n \| = 0.
        \end{equation}
        From this, \eqref{eq:y-to-boundary} and  the triangle inequality it follows that $t_n\|a_n \| > \gamma/2$ for all $n$ large enough. So, for $\theta := \gamma / (2 s)$ and some $N_2 \in \mathbb{N}$:
        $$
            t_n > \theta > 0,\quad\forall n > N_2.
        $$
        From \eqref{eq:tn-an}  then follows that  $x_n\in \{ \bar x\}+C(0;A_\delta )$ for all $n$ large enough, and \eqref{eq:f-bound-bound} together with \eqref{eq:y-to-boundary} gives $f(x_n) > f(\bar x) + k\gamma/2$ for all $n$ large enough. Since $\psi_n \ge 0$,
        \[
        f(x_n)+\psi_n(y_n)>f(\bar x)+k\gamma/2
        \]
        for large $n$, which
        contradicts \eqref{eq:f-psi-to-inf}, and the claim is verified.

        \smallskip

        So,  for arbitrary large $n$'s by (P2) we have $p_n\in\partial f(x_n)$ and \eqref{eq:f-psi-to-inf} implies (because of $\psi_n\ge 0$) that $f(x_n) \le f(\bar x) + 1/n < f(\bar x) + \varepsilon$ for large $n$'s. Since, $f(x_n) > f(\bar x) - \varepsilon$, see \eqref{eq:f-eps-bound}, the proof is completed.
    \end{proof}

    \section{Proof of Theorem~\ref{thm:DualCLI}}
    \label{sec:proof}
We will again follow the route traced in  \cite{ACL}.

Fix $\varepsilon \in (0,1)$ and let $\delta_1 > 0$ be so small that $4\delta_1 < \varepsilon$ and
$$
    r + 2\delta_1 < \inf f(A_{\delta_1}).
$$
The latter is possible due to \eqref{eq:cled-cond}. Let us denote for short $r_1 := r + \delta_1$. Let $\delta\in(0,\delta_1)$ be so small that
$$
   \Delta := \max\{(r_1-f(a)) (1-\delta),(r_1-f(a)) (1+\delta)\} < r_1 - f(a) + \delta_1,
$$
and
$2\delta |r_1-f(a)|<\eps/4$.

Consider the Banach space $\tilde X := X\times\mathbb{R}$ equipped with the norm $\|(x,r)\| := \max\{\|x\|,|r|\}$, the point $\tilde a := (a,0)$ and the set $\tilde A := A\times \{1\}=(A,1)$. Note that with this norm chosen, $\tilde A_\delta =A_\delta\times[1-\delta,1+\delta]$.

Consider the function $\widetilde f: \tilde X \to \Ri$ defined as
$$
    \widetilde f(x,t) := f(x) -(r_1-f(a)) t.
$$
We have that $\displaystyle \inf \widetilde f(\tilde A_\delta) = \inf f(A_\delta) - \sup_{t\in [1-\delta,1+\delta]} (r_1-f(a))t =\inf f(A_\delta) - \Delta$. Since $\delta_1$ and $\delta$ were chosen so that $\inf f(A_\delta) \ge \inf f(A_{\delta_1}) > r_1 + \delta_1$, and $-\Delta > f(a) - r_1 - \delta_1$, we have that
$$
  \inf  \widetilde f(\tilde A_\delta) > f(a) = \widetilde f(\tilde a).
$$

Let us take $r'\in (f(a) , \inf \widetilde f(\tilde A_\delta))$ such that $r'<f(a)+\delta$.

From Theorem~\ref{thm:primal-cled}, applied for $\widetilde f$, $\tilde A_\delta$ and $r'$,  there exists $(\bar x, \bar t)\in [\tilde a, \tilde A_\delta]$ such that

\begin{equation}
\label{eq:est r'}
\widetilde f(\bar x,\bar t) \le r'
\end{equation}
 and
$$
   \widetilde f^-((\bar x,\bar t), \tilde A_\delta - \tilde a) \ge r'-f(a)>0.
$$
Now from Proposition~\ref{prop:bridge},  applied for $\widetilde f$, $\tilde A_\delta - \tilde a$ and $\delta >0$, there is $(\xi,s) \in B^\circ((\bar x, \bar t);\delta)$ such that $\tilde f(\xi,s) < \tilde f(\bar x,\bar t) + \delta$, and, recalling (P6), there is $p\in \partial f(\xi)$ such that
$$
   \inf  (p,f(a)-r_1)(\tilde A - \tilde a ) > - \delta.
$$
We will check that $\xi \in X$ and $p\in X^*$ satisfy the conclusions. First, the above immediately yields $\inf p(A)-p(a) > r_1 - f(a) - \delta > r - f(a)$ and we have \eqref{eq:cl-p-A}.

Next, by (\ref{eq:est r'}) we have $\widetilde f(\xi,s)< \widetilde f(\bar x,\bar t)+\delta\le r'+\delta<f(a)+2\delta < f(a)+\eps/2$ and it follows that
\[
f(\xi)=\widetilde f(\xi,s)+(r_1-f(a))s< f(a)+\eps/2+(r_1-f(a))s.
\]
Since $\bar t\in [0,1+\delta]$, we have that $s\in [-\delta,1+2\delta]$. If $r_1\ge f(a)$ then $(r_1-f(a))s\le r_1-f(a)+2\delta |r_1-f(a)|\le r-f(a)+\delta +2\delta |r_1-f(a)|< r-f(a) +\eps/2$, and if $r_1< f(a)$, then $(r_1-f(a))s\le \delta |r_1-f(a)|< \eps/2$.
Hence,
\[
f(\xi)< \max\{ f(a),r\} + \eps,
\]
and we have \eqref{eq:f-xi-local}.

Finally, $\xi\in[a,A_\delta]_\delta\subset [a,A]_\varepsilon$. \qed

\section{Appendix}
 \label{sec:appendix}

Here we give a proof of Bishop-Phelps Lemma \ref{lem:bp}, which relies solely on notions of the underlying primal space $X$. The proof is based on a simple geometric fact, see Lemma~\ref{prop:app}, inspired by the techniques used for verifying the Long orbit or Empty value (LOEV) principle, see \cite{IZ}. First we will prove the following  straightforward preparatory result.

\begin{lem}
\label{lem:app}
Let $X$ be Banach space and $A\subset X$ be a non-empty, closed, bounded and convex set, such that $0\notin A$. Then there exists   $c>0$ such that
for any finite set of points $\{x_k\}_{k=1}^n\subset C(0,A)$ it holds that
\begin{equation}
\label{eq:estimate}
\left\|\sum_{k=1}^n x_k\right\|\ge c\sum_{k=1}^n\|\ x_n\|.
\end{equation}
\end{lem}

 \begin{proof}

 Let $\delta:=\inf_{a\in A}\|a\|>0$ and $L:=\sup_{a\in A}\|a\|<\infty$. Set $c:=\delta/L$. Let $x_k=t_ka_k$ for some $t_k\ge0$ and $a_k\in A$, $k=1,\dots,n$. If $t_k=0$ for all $k=1,\dots, n$,  (\ref{eq:estimate}) is trivial. Else, set $t:=\sum_{k=1}^n t_k>0$ and consider $t'_k:=t_k/t\ge 0$. By convexity of $A$ we have that $ \sum_{k=1}^n t_k'a_k=a\in A$, and then
 \[
 \left\|\sum_{k=1}^n x_k\right\|=\left\|\sum_{k=1}^n t_ka_k\right\|=t\left\|\sum_{k=1}^n t'_ka_k\right\|=t\|a\|\ge t\delta .
 \]
 On the other hand,
 \[
\frac{\delta}{L}\sum_{k=1}^n \|x_k\|=\frac{\delta}{L}\sum_{k=1}^{n}t_k\|a_k\|\le \delta \sum_{k=1}^{n}t_k=\delta t.
 \]
By combining both, one gets \eqref{eq:estimate}.
 \end{proof}

\begin{lem}
\label{prop:app}
Let $X$ be a Banach space and $B\subset X$ be a non-empty, closed, bounded and convex set. Let $V\subset[0,B]$ be a  closed set such that $0\in V$. 

If $V\cap B=\varnothing$, then 
\[
\exists x_0\in V \text{ such that } \forall  y\in(\{x_0\}+C(0;B))\setminus\{x_0\} \text{ it follows } y\notin V.
\]
\end{lem}
\begin{proof}
If is enough to prove that if
\begin{equation}
\label{eq:cond}
\forall x\in V\setminus B, \quad \exists y\in\{x\}+C(0;B) \text{ such that } y\in V\setminus\{x\},
\end{equation}
then $V\cap B\neq\varnothing$.

Let us assume the contrary, i.e. \eqref{eq:cond} holds but $V\cap B=\varnothing$. This implies $0\notin B$ so there exists $\e>0$ such that $0\notin \ol{B_{\e}}$.

 Starting with  $x_0=0$ we construct by iteration a sequence of points $\{x_n\}\in V$. Therefore $x_n\in V\setminus B$ and from ($\ref{eq:cond}$) there exists $y\neq x_n$ such that $y\in V\cap(\{x_n\}+C(0;B))$.  Since $C(0;B)\subset C(0;B_{\e})$, we have
\begin{equation}
\label{eq:mu_n}
\nu_n:=\sup\{\|y-x_n\|:y\in V\cap(\{x_n\}+C(0;B_{\e}))\}>0.
\end{equation}
We take
\begin{equation}
\label{eq:x_n def}
x_{n+1}\in V\cap(\{x_n\}+C(0;B_{\e})) \text{ such that } \|x_{n+1}-x_n\|>\nu_n/2.
\end{equation}
Applying Lemma \ref{lem:app} for the set $\ol{B_{\e}}$ and the points $\{x_{k+1}-x_k\}_{k=0}^{n-1}\subset C(0;B_\eps)$, for a finite constant $c>0$ we get  that
\[
\|x_n\|=\left\|\sum_{k=0}^{n-1}(x_{k+1}-x_k)\right\|\ge c\sum_{k=0}^{n-1}\|x_{k+1}-x_k\|.
\]
Since $V$ is bounded and $x_n\in V $ for all $n$, the series $\sum_{k=0}^{n-1}\|x_{k+1}-x_k\|$  converges. Therefore $(x_n)_{n=0}^{\infty}$ is a Cauchy sequence and by the closedness of $V$ we obtain that 
\begin{equation}
\label{eq:bar x}
\lim_{n\to\infty}x_n=:\bar{x}\in V.
\end{equation}

Now ($\ref{eq:x_n def}$) implies that $\nu_n\to0$.

On the other hand,  from ($\ref{eq:cond}$) it follows the existence of $\bar{y}\in V$, $\bar{y}\neq\bar{x}$ such that $\bar{y}=\bar{x}+\bar{t}\bar{b}$ for some $\bar{t}>0$ and $\bar{b}\in B$. For every $n$ large enough (\ref{eq:bar x}) yields $\|\bar{y}-x_n-\bar{t}\bar{b}\|<\e\bar{t}$, which means that $(\bar{y}-x_n)/\bar t\in B_{\e}$ or $\bar{y}\in V\cap(\{x_n\}+C(0;B_{\e}))$. Then from ($\ref{eq:mu_n}$),  $\nu_n\ge\|\bar{y}-x_n\|>\|\bar{y}-\bar{x}\|/2>0$ for all $n$ large enough. The latter contradicts $\nu_n\to 0$. The proof is then completed.
 \end{proof}

{\textbf{Proof of Lemma \ref{lem:bp}}}. Put $M':=M-a$ and $B':=A-a$, and observe that $0\notin B'$ and $0\in M'$. Set $V:=M'\cap C(0;B')$, thus $0\in V$. Furthermore, $V+a=M\cap(\{a\}+C(0;A-a))\subset M$ and thus $V$ is bounded.

Since $0\notin B'$,   there is $t>0$ such that for $B:=tB'$ one has  $V\cap B=\varnothing$ and $V\subset[0,B]\subset C(0;B')$, thus $V=M'\cap [0,B]$. The latter is a closed set as it was explained at the beginning of Section \ref{sec:primal}. Therefore,  the conditions of Lemma~\ref{prop:app} are fulfilled for the sets $B$ and $V$. By Lemma~\ref{prop:app} one gets a $x_0\in V$ such that $\forall y\in\{x_0\}+C(0;B)=\{x_0\}+C(0;B'), \ y\neq x_0$ it holds that $y\notin V$. In the end, we obtain a point $\bar{x}:=x_0+a\in M\cap(\{a\}+C(0;A-a))$ such that $M\cap(\{\bar{x}\}+C(0;A-a))=\{\bar{x}\}$. The proof is completed. \qed

\end{document}